\documentclass{article}
\pdfoutput=1
\usepackage{arxiv}
\usepackage[utf8]{inputenc} % allow utf-8 input
\usepackage[T1]{fontenc}    % use 8-bit T1 fonts

\usepackage{combelow}

\usepackage{hyperref}       % hyperlinks
\usepackage{url}            % simple URL typesetting
\usepackage{booktabs}       % professional-quality tables
\usepackage{amsfonts}       % blackboard math symbols
\usepackage{nicefrac}       % compact symbols for 1/2, etc.
\usepackage{microtype}      % microtypography

\usepackage{booktabs}
\usepackage{subeqnarray}
\usepackage{caption}
\usepackage{subcaption}
\usepackage{mleftright}
\usepackage{multirow}
\usepackage{lipsum,cite}

\usepackage{graphicx}
\usepackage{epstopdf}
\usepackage{algorithmic}
\usepackage{algorithm}
\usepackage{soul}

\RequirePackage[mathscr]{eucal}

\usepackage{amsmath}
\usepackage{amssymb}
\usepackage{amsfonts}
\usepackage{amsthm}
\usepackage{graphicx}
\usepackage{color}

\newtheorem{exercise}{Exercise}
\newtheorem{exampl}{Example}
\newtheorem{theorem}{Theorem}[section]
\newtheorem{corollary}{Corollary}[theorem]
\newtheorem{lemma}[theorem]{Lemma}

\def\bq{\begin{quotation}}
\def\eq{\end{quotation}}

\definecolor{darkgreen}{rgb}{0.2,0.5,0.2}

\def\o{\omega}

\newcommand{\V}[1]{ \mathbf{#1} }    % vector

\newcommand{\Vo}{\V{0}}

\newcommand{\Ve}{\V{e}}

\newcommand{\Vm}{\V{m}}

\newcommand{\Vp}{\V{p}}

\newcommand{\Vr}{\V{r}}

\newcommand{\Vu}{\V{u}}

\newcommand{\Vx}{\V{x}}
\newcommand{\Vy}{\V{y}}

\newcommand{\M}[1]{ \mathbf{#1} }  % matrix

\newcommand{\MA}{\M{A}}
\newcommand{\MB}{\M{B}}
\newcommand{\MC}{\M{C}}
\newcommand{\MD}{\M{D}}
\newcommand{\ME}{\M{E}}
\newcommand{\MF}{\M{F}}
\newcommand{\MG}{\M{G}}

\newcommand{\MI}{\M{I}}
\newcommand{\MJ}{\M{J}}

\newcommand{\ML}{\M{L}}

\newcommand{\MR}{\M{R}}

\newcommand{\MU}{\M{U}}
\newcommand{\MV}{\M{V}}
\newcommand{\MW}{\M{W}}

\newcommand{\MY}{\M{Y}}
\newcommand{\MZ}{\M{Z}}

\newcommand{\Rn}[1]{\mathbb{R}^#1}
\newcommand{\Rmn}[2]{\mathbb{R}^{#1 \times #2}}

\newcommand{\Cmn}[2]{\mathbb{C}^{#1 \times #2}}

\def\2nm#1{\|#1\|_2}

\def\Ra#1{\mathrm{Range}(#1)}

\newcommand{\ars}[1]{\left[ \begin{array}{#1}}
\newcommand{\are}{\end{array} \right] }
\newcommand{\oars}[1]{\begin{array}{#1}}
\newcommand{\oare}{\end{array}}
\newcommand{\rars}[1]{\left( \begin{array}{#1}}
\newcommand{\rare}{\end{array} \right) }

\newcommand{\eqs}{\begin{eqnarray}}
\newcommand{\eqe}{\end{eqnarray}}
\newcommand{\eqsn}{\begin{eqnarray*}}
\newcommand{\eqen}{\end{eqnarray*}}

\def\defs{\begin{definition}}
\def\defe{\end{definition}}
\def\teos{\begin{theorem}}
\def\teoe{\end{theorem}}
\def\prfs{\begin{proof}}
\def\prfe{\end{proof}}
\def\exas{\begin{exampl}}
\def\exae{\end{exampl}}
\def\excs{\begin{exercise}}
\def\exce{\end{exercise}}
\def\cors{\begin{corollary}}
\def\core{\end{corollary}}

\newcommand{\ens}{\begin{enumerate}}
\newcommand{\ene}{\end{enumerate}}

\newcommand{\its}{\begin{itemize}}
\newcommand{\ite}{\end{itemize}}

\newcommand{\des}{\begin{description}}
\newcommand{\dee}{\end{description}}

\def\wh{\widehat}

\def\wt{\widetilde}

\newcommand{\mb}{\mathbf}
\newtheorem{remark}{Remark}

\hypersetup{colorlinks=true, linktoc=all,  linkcolor=blue}

\def\o{\omega}

\newcommand{\bfsfp}{\mbox{\boldmath$\mathsf{p}$} }

\newcommand{\rdir}{\Vr}  %right direction
\newcommand{\ldir}{\pmb{\ell}} %left direction
\newcommand{\Rdir}{\MR}  %right direction
\newcommand{\Ldir}{\ML} %left direction
\newcommand{\MVrand}{\widetilde{\MV} } %rand right subspace
\newcommand{\MWrand}{\widetilde{\MW}}  %rand left subspace
\newcommand{\Psirand}{\widetilde{\pmb{\Psi}}}  %rand left subspace

\newcommand{\obs}{\wh{\Vm}}

\newcommand{\detsol}{\wh{\MZ}}

\newcommand{\randsrcsol}{\wt{\MY}}
\newcommand{\randdetsol}{\wt{\MZ}}

\newcommand{\MAsym}{\wt{\MA}}  %symmetric A(p)
\newcommand{\MBt}{\wt{\MB}}  %symmetric A(p)
\newcommand{\MDeltasym}{\wt{\Delta \MA}^{(i)}} %symmetric DeltaA
\newcommand{\MUtil}{\wt{\MU}} %U in DeltaA=DeltaA_{jj} U*Y^T
\newcommand{\MYtil}{ ( \wt{\MU}^{(i)} )^T } %Y in DeltaA=U*Y^T

\newcommand{\EigvecsUnit}{\wt{\pmb{\Phi}}} %Each column of Phi is an Eigenvector of Laplacian -unitary
\newcommand{\eigvecUnitD}{\wt{\phi}_{i_x i_y}^{k_x k_y}}%2D Eigenvector of Laplacian
\title{Randomization for the Efficient Computation of 
Parametric Reduced Order Models for Inversion}

\author{
 Selin Aslan \\
  Advanced Photon Source\\
  Argonne National Laboratory\\
  Lemont, IL 60439\\
  \texttt{saslan@anl.gov} \\
   \And 
Eric de Sturler \\
  Department of Mathematics and \\
  Computational Modeling and Data Analytics Division\\
  Academy of Integrated Science\\
  Virginia Tech, Blacksburg, VA 24061\\
  \texttt{sturler@vt.edu} \\
   \AND
   Serkan Gugercin \\
  Department of Mathematics and \\
  Computational Modeling and Data Analytics Division\\
  Academy of Integrated Science\\
  Virginia Tech, Blacksburg, VA 2406\\
   \texttt{gugercin@vt.edu} \\
}

\begin{document}
\maketitle

\begin{abstract}
Nonlinear parametric inverse problems appear in many applications. In this work, we focus on 
diffuse optical tomography (DOT) in medical image
reconstruction, in which we aim to recover an unknown image of interest (defined by a set of parameters), such as cancerous
tissue in a given medium, using a mathematical (forward) model. The forward model in DOT is
a diffusion-absorption model for the photon flux. The main computational bottleneck in such inverse
problems is the repeated evaluation of the large-scale forward model. For DOT, this
corresponds to solving large linear systems for each source and frequency at each optimization
step. 
 
In addition, as Newton-type methods are the method of choice for these problems, one needs to solve additional
 linear systems with the adjoint to
 efficiently compute derivative information.
As rapid advances in technology allow for large
numbers of sources and detectors, these problems become computationally prohibitively
expensive. In the past, the use of reduced order models (ROM) has been proposed to drastically
reduce the size of the linear systems solved in each optimization step, while still solving the
inverse problem accurately. However, as the number of sources and detectors increases, even
the construction of the ROM bases  incurs a substantial cost.
Interpolatory model reduction, when performed in the setting of full matrix-interpolation to match the full parameter gradient matrix, requires the solution of large linear systems for all sources and
frequencies as well as for all detectors and frequencies for each interpolation point 
in parameter space to build a candidate basis for the ROM projection space. As this candidate
basis numerically has low rank, this construction is followed by 
a rank-revealing 
factorization that typically reduces the number of vectors in
the candidate basis substantially. 
Since only the ROM projection space matters, not its basis, 
we propose to use randomization to approximate
this basis efficiently and drastically reduce 
the number of large linear solves for constructing the global ROM basis for DOT. 
We also provide a detailed analysis for the low-rank structure of the candidate basis for our 
problem of interest. Even though we focus on the DOT problem here,
the ideas presented are relevant to
many other large scale inverse problems and optimization problems.
\end{abstract}

% keywords can be removed

\section{Introduction}\label{sec:intro}
Nonlinear inverse problems appear in many applications for identification and localization of anomalous regions, such as finding tumors in the body, luggage screening, and finding contaminant pools in the earth. In these applications, we aim to recover an unknown image of a region of interest in a given medium using a mathematical model (the forward model) combined with measurements
\cite{hansen1998rank,Epstein2008,Tarantola2005}.

In DOT, the inverse problem we focus on in this paper, the forward model is described by a large-scale discretized partial differential equation (PDE), namely a diffusion-absorption model for the photon flux. The main computational bottleneck is the repeated evaluation of this forward model and its derivatives to recover the desired image. The number of sources and detectors may be a thousand or more in 3D, combined with multiple frequencies, and
the resulting computational cost can be prohibitively expensive. Hence, we need new computational techniques
that characterize the medium quickly and efficiently for such inverse problems.

As shown in \cite{StuGuKilChatBeatOCon}, the cost function and gradient evaluation in DOT correspond to evaluating the frequency response (transfer) function of a linear dynamical system at prespecified frequencies. Therefore, the problem is perfectly suited for interpolation-based model reduction. This framework, i.e.,
the use of interpolatory parametric reduced models as surrogates for the full forward model to reduce the size of the linear systems solved in each optimization step, 
has been considered in DOT while approximating both the cost functional and the Jacobian accurately \cite{StuGuKilChatBeatOCon}. This approach significantly reduces the cost of the inversion by drastically reducing the computational cost of solving the forward problems. However, to compute the model reduction bases used in the ROM, many large linear systems still need to be solved, followed by a rank-revealing factorization.

This rank-revealing factorization in building the ROM basis reveals that the interpolatory method for ROM construction to guarantee full cost function and gradient interpolation
solves many more linear systems than needed \cite{StuGuKilChatBeatOCon,OConKilStuGug}. In other words, using ROMs reduces the size of the linear systems solved in each step of the optimization, but it still requires the solution of many large systems to build a ROM basis that enforces full matrix interpolation. In \cite{AslanStuKilmer}, we show that randomization in DOT can reduce the number of large linear system solves. Hence, in this paper, we combine these two approaches to obtain an effective and computationally highly efficient approach for nonlinear parametric inversion. We employ randomization to capture essentially the same ROM space  as the standard approach at much lower cost via sampling.

In Section \ref{Sec:RandROMBackground}, we briefly review the DOT problem. In Section \ref{Sec:ROM}, we give the system-theoretic interpretation and notation for DOT and a brief overview of interpolatory model reduction. In Section \ref{Sec:RandROM}, we introduce a new randomized approach to generate ROMs efficiently. This approach is based on the fact that a spanning set for the range of a low rank matrix can be computed efficiently by random sampling \cite{HalMartTropp}. In Section \ref{sec:TheorJust}, we provide a theoretical justification for exploiting low rank structure in the reduction basis, and we connect our approach, using randomization to compute the interpolatory model reduction bases, to tangential interpolation. Numerical results are given in Section \ref{Sec:NumExp} for 3D problems. 
For large 3D problems, iterative solvers are needed to generate the candidate ROM basis, and for that reason we
use the number of linear systems to be solved as a proxy for computational cost and we consider each right hand side as a separate system. Conclusions and future work are outlined in Section \ref{sec:conc}.

Before we continue, we point out that different model reduction techniques have been employed in a wide variety of inverse problems; see, e.g., \cite{Druskin2011solution,Galbally2009, lieberman2013hessian, Lieberman2010, Wang:2005b, borcea2012model,cstefuanescu2016model,cui2015data,liu2005inverse,manzoni2014reduced,benner2014model} and the references therein. However, as we explain below, interpolatory model reduction methods \cite{AntBG20,Ant2010imr,BeattGug_bookCh7_RatInt} are perfectly suited for the DOT problem and have proven very effective in this setting \cite{StuGuKilChatBeatOCon,OConKilStuGug}. Hence, our focus here is on interpolatory approaches for DOT. Our goal is not to compare different model reduction techniques for inverse problems in general.

\section{Diffuse Optical Tomography}\label{Sec:RandROMBackground} DOT is a non-invasive, low cost alternative for breast and brain imaging compared with X-Ray and MRI. In DOT, near infra-red light from an array of sources is transmitted through the medium and measured with an array of detectors. Here, we assume that the diffusion coefficient is known, and  we use measurements and the forward model to recover the absorption coefficient of the medium. The absorption coefficient can be used to distinguish healthy tissue from tumors \cite{arridge}.

We consider the diffusion model, posed in the frequency domain, for the photon flux $\phi(\mb{x})$ driven by an input source $g(\mb{x})$. The forward model for DOT is given by
\begin{align}\label{Eq: PDE2}
&-\nabla \cdot (D(\Vx)  \nabla\phi(\Vx))+ \mu(\Vx)\phi(\Vx)+ \frac{ \imath\omega}{\nu}\phi(\Vx)= g(\mb{x}), &\\  \nonumber
&\text{for} \quad \Vx=(x_1,x_2,x_3)^T \;\text{and}\;  -a<x_1<a,\:\:\:-b<x_2<b, \:\:\:0<x_3<c, &\\  \nonumber
&\phi(\Vx)=0\:\:   \text{if}\:\: 0\leq x_3 \leq c\:\:\: \text{and} \:\:\: \text{either}\:\:\: x_1=\pm a, \:\:\:\text{or}\:\:\:\: x_2=\pm b, &\\  \nonumber
&0.25\phi(\Vx)+\frac{D(\Vx)}{2}\frac{\partial \phi(\Vx)}{\partial\xi}= 0  \:\:  \text{for}\:\:\:x_3=0,\:\:or\:\: x_3=c,\nonumber
\end{align}
where $D(\mb{x})$ is the diffusion coefficient, $\mu(\mb{x})$ the absorption coefficient, $\omega$ is the frequency modulation of light, $\nu$ is the speed of light in the medium,
and $\xi$ is the outward unit normal on the boundary \cite{arridge}.

The Parametric Level Set (PaLS) approach \cite{Aghasi_etal11} has been used  to reconstruct complex geometries, and provide regularization to compensate for the influence of noise and ill-posedness of DOT \cite{Aghasi_etal11, StuGuKilChatBeatOCon, AslanStuKilmer}. We parameterize the absorption field  $\mu(\mb{x}; \Vp)$ using PaLS and reduce the dimension of the parameter space. Hence, we only solve for a modest number of parameters, $\Vp \in \Rn{{n_p}}$ that describe the shape of potential anomalies (tumors), rather than solving for absorption at every grid point.

Let $n_s$ be the number of sources on the top, $n_d$ be the number of detectors on the bottom, and $n_\o$ be the number of frequencies to generate the DOT data. The discretization of (\ref{Eq: PDE2}) can be done by finite element or finite difference techniques. Let the rows of $\MC \in  \Rmn{n_d}{n}$ correspond to the detectors, the columns of $\MB \in \Rmn{n}{n_s}$ correspond to the sources, and let $\obs(\omega;\Vp) $ be the vector of detector outputs. After discretization of (\ref{Eq: PDE2}), measurements at the detector locations satisfy
\begin{equation} \label{Eq:FOMtransferFun}
  \obs(\omega;\Vp) =\pmb{\Psi}\!\left(\omega;\Vp \right)\,\hat{\Vu}(\omega)\quad \mbox{where}\quad
%  \pmb{\Psi}(\omega;\Vp)=\MC\hat{\Vy} \  \mbox{and} \
   \pmb{\Psi}(\omega;\Vp)=\MC \left(\frac{\imath\omega}{\nu}\,\ME\, +\MA(\Vp)\right)^{-1}\MB,
\end{equation}
$\MA(\Vp) \in \Rmn{n}{n}$ results from a finite difference discretization of the diffusion and absorption terms, and $\ME \in  \Rmn{n}{n}$ is the identity except for zero rows corresponding to points on the boundaries $x_3=0$ and $x_3=c$ in (\ref{Eq: PDE2}). Thus, $\ME$ is singular. For a given frequency $\o$ and parameter vector $\Vp$, $ \pmb{\Psi}(\omega;\Vp)$ gives the map from inputs $\hat{\Vu}(\o)$ to outputs $\obs(\omega;\Vp)$; this is known as the transfer (frequency response) function. Combining all the predicted observation vectors, we obtain the complex matrix
\begin{equation}
\mathbb{M}(\Vp) = [\obs_1(\omega_1;\Vp),\, \ldots,\,\obs_{n_s}(\o_1;\Vp),\,\obs_1(\o_2;\Vp), \ldots,
\obs_{n_s}(\omega_{n_\omega};\Vp) ]
\end{equation}
of dimension $n_d \times (n_s \cdot n_\o)$. Given the empirical data $\mathbb{D}$, we solve the nonlinear least squares problem
\begin{equation}\label{Eq:NonlinLS}
  \wh{\Vp} := \arg \: \min\limits_{\Vp} \|\mathbb{M}(\Vp) - \mathbb{D} \|_{F}^2.
\end{equation}
The solution $\wh{\Vp}$ identifies the absorption field and thus the anomaly. The PaLS representation regularizes the problem, so no further regularization is needed \cite{Aghasi_etal11}. Evaluating the objective function at $\Vp$ requires solving the systems
\begin{equation}\label{Eq:Srcsols}
\left(\frac{\imath\o_j}{\nu}\,\ME\, +\MA(\Vp)\right)\wh{\MY}=\MB,
\end{equation}
for each frequency, leading to $n_s\cdot n_\omega$ large linear systems.
For solving the minimization problem (\ref{Eq:NonlinLS}), we use the Trust region algorithm with REGularized model Solution (TREGS) \cite{StuKil11c} that has been proven very effective for problems of this type.
For Newton-type algorithms, it is also necessary to construct the Jacobian of $\pmb{\Psi}(\Vp)$,
\begin{equation} \label{eq:jacob}
\MJ(\Vp)=-\MC\left( \frac{\imath \o_j}{\nu}\ME +\MA(\Vp) \right)^{-1}\frac{\partial}{\partial \Vp}\MA(\Vp)\left(\frac{\imath \o_j}{\nu}\ME +\MA(\Vp) \right)^{-1}\MB,
\end{equation}
where $\frac{\partial}{\partial \Vp_k}\MA(\Vp)$ is diagonal and inexpensive to compute. Evaluating $\MJ$ using the co-state approach~\cite{Vogel2002} requires solving the systems
\begin{equation}\label{Eq:Detsols}
\left(\frac{\imath\o_j}{\nu}\,\ME\, +\MA(\Vp)\right)^T \detsol=\MC^T,
\end{equation}
for each frequency, leading to an additional $n_d\cdot n_\omega$ adjoint systems for the detectors. As a result, standard optimization approaches to solve this inverse problem require $O(10^3-10^4)$ large linear system solves at each optimization step. The size of a realistic linear system is at least $O(10^6)$. Repeated evaluations of (\ref{Eq:Srcsols}) and (\ref{Eq:Detsols}) lead to a computational bottleneck for solving the inverse problem. We will resolve this issue by incorporating interpolatory model reduction into the process.

\section{Interpolatory Model Reduction}\label{Sec:ROM} Projection-based parametric model reduction is widely used to replace a parametric dynamical system of high dimension, $n$, by one of much lower dimension, $r \ll n$, to allow rapid yet accurate simulation over the range of parameters \cite{BeattGug_bookCh7_RatInt,BGW2015,BenCOW17,hesthaven2016certified, AntBG20,QuaMN16}. To briefly explain the construction of ROM, we consider the time domain representation of the frequency domain equation (\ref{Eq:FOMtransferFun}) given by
\begin{equation}\label{Eq:ModelDynSys}
  \frac{1}{\nu}\ME\,  \dot{\Vy}(t;\Vp)  =-\MA(\Vp)\Vy(t;\Vp) +
	\MB\Vu(t)\quad\mbox{with}\quad
	\Vm(t;\bfsfp) =\MC\Vy(t;\Vp).
\end{equation}
We seek to replace this high dimensional dynamical system by a reduced order model
\begin{equation}\label{Eq:ROMDynSys}
  \frac{1}{\nu}\ME_r\,  \dot{\Vy_r}(t;\Vp)  =-\MA_r(\Vp)\Vy_r(t;\Vp) +
	\MB_r\Vu(t)\quad\mbox{with}\quad
	\Vm_r(t;\bfsfp) =\MC_r\Vy_r(t;\Vp),
\end{equation}
with the associated reduced transfer function
\begin{equation}\label{Eq:TransferFunc}
\pmb{\Psi}_r(\omega;\Vp)  =  \MC_r\left( \frac{\imath \o}{\nu}\ME_r +\MA_r(\Vp)\right)^{-1} \MB_r \in {\mathbb{C}^{n_d \times n_s}},
\end{equation}
where
\begin{align}\label{Eq:ROMeq}
\ME_r= \MW_r^T \ME \MV_r,  \:\:\:\:  \MB_r =\MW_r^T\MB, \:\:\:\: \MC_r = \MC \MV_r, \:\:\:\: \MA_r(\Vp) = \MW_r^T \MA(\Vp) \MV_r,
\end{align}
and $\MV_r \in\Cmn{n}{r}$ and $\MW_r \in\Cmn{n}{r}$ are the computed projection basis matrices. To obtain a high-fidelity approximation to $\pmb{\Psi}(\omega;\Vp)$ over all parameters and frequencies
of interest, we need to construct appropriate matrices $\MW_r$ and $\MV_r$.
Several parametric model reduction methods exist to select $\MW_r$ and $\MV_r$; see, e.g., \cite{BGW2015,benner2017model,BenCOW17,hesthaven2016certified,BuiThanh2008, BuiThanh2008_AIAA,feng2019new,feng2020error,beattie2019sampling,chen2019robust} and the references therein.
In DOT, the function and Jacobian evaluations correspond to evaluating, respectively, the transfer function $\pmb{\Psi}(\omega_j;\Vp)$
and its derivatives at parameter points that arise in the minimization (\ref{Eq:NonlinLS}).
Therefore, a natural choice is to use interpolatory model reduction to construct $\MW_r$ and $\MV_r$ \cite{bond2005pmo,BauBBG09,gunupudi2003ppt,FenB08,AntBG20}.
We follow well-known results regarding interpolatory parametric model reduction in \cite[Theorem 4.1-4.2]{BauBBG09} to generate the ROM using Algorithm \ref{Alg:ROMallsrcdet} (below). The ROM and its derivatives match the full order model and its derivatives {\em exactly} at the interpolation points in frequency and parameter space, $(\o_j, \pmb{\pi}_i)$ for $j=1,\cdots, n_\o$ and $i=1,\cdots, n_k$, that is,
\begin{align}
  \pmb{\Psi}(\omega_j;\pmb{\pi}_i)&=\pmb{\Psi}_r(\omega_j;\pmb{\pi}_i), \nonumber \\
  \nabla_{\Vp}\pmb{\Psi}(\omega_j;\pmb{\pi}_i)&=\nabla_{\Vp}\pmb{\Psi}_r(\omega_j;\pmb{\pi}_i,\label{Eq:ROMhermite} ~\mbox{and}\\
  \pmb{\Psi}^\prime(\omega_j;\pmb{\pi}_i)&=\pmb{\Psi}_r^\prime(\omega_j;\pmb{\pi}_i),\nonumber
\end{align}
where $'$ denotes the derivative with respect to $\omega$.
It follows from the definition of the Jacobian $\MJ(\Vp)$
defined  in \eqref{eq:jacob} that matching the parameter gradient of the transfer function, i.e., matching $\nabla_{\Vp}\pmb{\Psi}(\omega_j;\pmb{\pi}_i)$, 
is equivalent to matching the Jacobian $\MJ(\Vp)$  at $\Vp = \pmb{\pi}_i$. 
Therefore, in terms of the underlying optimization problem,
the cost function and its parameter gradients are exactly matched at the sampled points; and thus
the optimization algorithm does not see the difference between the reduced order model and the full order model at the selected interpolation points and parameter samples.

\begin{algorithm}
\caption{\cite{StuGuKilChatBeatOCon}}\label{Alg:ROMallsrcdet}
Given parameter interpolation points, $\pmb\pi_1,\cdots,\pmb\pi_{n_k}$, and  frequency interpolation points, $\omega_1,\cdots,\omega_{n_{\omega}}$
\begin{itemize}
  \item[(1)] Generate candidate bases for $i=1,\cdots {n_k}$\\
 $ \MV^{(i)}=\left[\left( \frac{\imath\o_1}{\nu}\ME-\MA(\pmb\pi_i)\right)^{-1}\MB, \cdots,  \left(\frac{\imath\o_{n_\o}}{\nu}\ME-\MA(\pmb\pi_i)\right)^{-1}\MB		\right]$\\
  $ \MW^{(i)}=\left[\left( \frac{\imath\o_1}{\nu}\ME-\MA(\pmb\pi_i)\right)^{-T}\MC^T, \cdots,  \left(\frac{\imath\o_{n_\o}}{\nu}\ME-\MA(\pmb\pi_i)\right)^{-T}\MC^T 		\right]$
  \item[(2)] Let $\MV = [\MV^{(1)},\MV^{(2)},\cdots,\MV^{({n_k})}]$ and $\MW = [\MW^{(1)},\MW^{(2)},\cdots,\MW^{({n_k})}]$.
 \item[(3)]Use SVD or rank-revealing QR factorizations of
 $\MV$ and
 $\MW$
 to compute the ROM bases $\mathbf{Q}_\MV$ and  $\mathbf{Q}_\MW$, respectively. To retain underlying symmetry, set $\MV_r \leftarrow [\mathbf{Q}_\MV, \ \mathbf{Q}_\MW]$ (one-sided projection) to obtain the global basis $\MV_r\: (=\MW_r)$.
\end{itemize}
\end{algorithm}
After the global basis $\MV_r$ has been computed using Algorithm \ref{Alg:ROMallsrcdet}, we can compute the reduced order model for every parameter $\Vp$ as in (\ref{Eq:ROMeq}) with $ \MW_r= \MV_r$.

In our DOT setting, $\MA(\Vp)=\MA_0 +\MA_1(\Vp)$ where $\MA_0$ is constant and
$\MA_1(\Vp)$ carries the parametric dependency. Since $\MA(\Vp) \in  \Rmn{n}{n}$ results from a finite difference discretization of the diffusion and absorption terms, and here we invert for absorption only, in our implementation $\MA_1(\Vp)$ is diagonal, allowing efficient computation of $\MA_r(\Vp) $ during optimization. We refer to \cite{StuGuKilChatBeatOCon} for details. The selection of interpolation points/parameter samples for our test problems is discussed in Section \ref{Sec:NumExp}.

In order to interpolate the matrix-valued transfer function $\pmb{\Psi}(\omega;\Vp)$
and the matrix-valued Jacobian 
$\MJ(\Vp)$ at the sample points,
Algorithm \ref{Alg:ROMallsrcdet} requires  solving systems (\ref{Eq:Srcsols}) and (\ref{Eq:Detsols}) for every parameter and frequency interpolation point to generate candidate bases, leading to $n_k n_\o n_s+n_k  n_\o n_d$ large linear systems. Since our main interest lies in large 3D problems, where iterative solves are necessary, we use the
number of large linear systems as a proxy for computational cost and count every right hand side as a separate linear system to solve.
However, where (sparse) direct solves are feasible and cost-wise appropriate, it is important to note some linear systems have the same coefficient matrix, which requires only one (sparse) decomposition followed by triangular solves; i.e., solving with 
$\left( \frac{\imath\o_j}{\nu}\ME-\MA(\pmb\pi_i)\right)$
and its transpose for multiple systems requires only 
one factorization. 

If the number of sources and detectors is large, the construction of the ROM bases in Algorithm \ref{Alg:ROMallsrcdet} incurs a substantial cost at the start of the inversion. In fact, for DOT using parametric model order reduction (PMOR) \cite{StuGuKilChatBeatOCon}, this is currently the dominant cost, as the actual inversion is now quite fast. In addition, after generating a candidate basis, Algorithm \ref{Alg:ROMallsrcdet} uses a rank-revealing reduction to compute the ROM basis, which will bring in an additional computational cost, especially when the number of sources and detectors is large and the reduced model order is relatively large as well. This is typically the case in our application. Extensive numerical experiments illustrate that
this rank-revealing reduction step results in a much smaller dimensional 
basis \cite{StuGuKilChatBeatOCon}. We will take advantage of this fact to drastically reduce the cost of the DOT inversion process. To give some perspective on the (numerical) rank of 
the candidate basis versus the number of columns, consider the second 3D test problem in Section~\ref{Sec:NumExp} with 225 sources and detectors, 3 parameter sample points, and 4 frequency interpolation points. The full-matrix interpolation method for computing the global ROM basis
(as in Algorithm \ref{Alg:ROMallsrcdet}) requires solving 5400 large linear systems.
However, after a rank-revealing factorization only 1228 of those directions are used for the global ROM basis. This shows that we have computed a lot of redundant information.  
A partial solution to this problem, based on ideas from Krylov subspace
recycling \cite{parks2006rks,KilStu2006}
to compute the global basis more efficiently for DOT, was recently explored in \cite{OConKilStuGug}.

The approach we propose in this paper to significantly reduce the number of linear solves is
based on multiplying a (nearly) low rank matrix by a modest number of random vectors to exploit (and recover) the low rank structure via sampling \cite{HalMartTropp}.
In Section \ref{Sec:RandROM}, we employ randomization to capture essentially the same ROM projection 
space at much lower cost via sampling. Hence, we avoid first solving each linear system in (\ref{Eq:Srcsols}) and (\ref{Eq:Detsols}) and thus also reduce cost of the rank-revealing factorization.

\section{Randomization and Reduced Order Modeling}\label{Sec:RandROM}
In this section, we develop a more efficient algorithm to compute the model reduction basis by exploiting low rank structure in the candidate bases computed by Algorithm \ref{Alg:ROMallsrcdet}. The new algorithm significantly reduces the number of large linear solves for constructing the global ROM basis
for DOT by random sampling. Algorithm  \ref{Alg:ROMallsrcdet} computes the candidate bases 
$[ \MV^{(1)} \; \MV^{(2)} \: 
   \ldots \: \MV^{(n_k)} ]$ and 
$[ \MW^{(1)} \; \MW^{(2)} \; 
   \ldots \; \MW^{(n_k)} ]$. 
Following \cite{HalMartTropp}, we can approximate 
the range of these matrices by sampling.
Rather than sampling the full candidate bases, we sample each submatrix corresponding to a frequency and
parameter point individually 
to capture the range of that submatrix. 
If we capture the range of each submatrix, then 
the sum of these ranges contains the range of the 
candidate bases. We implement this approach using a modest number of
random linear combinations of the right-hand sides as
randomized sources and detectors and solve only for this modest number of
resulting right hand sides. 
Towards this goal, we solve
\begin{align}\label{Eq:RandROMsolves}
\left(\frac{\imath\o_j}{\nu}\,\ME\, +\MA(\pmb\pi_{i})\right)\randsrcsol=\MB\Rdir^{(i)}_j \ \quad \text{and} \quad \
\left(\frac{\imath\o_j}{\nu}\,\ME\, +\MA(\pmb\pi_{i})\right)^T \randdetsol=\MC^T\Ldir^{(i)}_j,
\end{align}
where each $\Rdir_j^{(i)} = [\rdir_1 \cdots \rdir_{\ell_s}]_j^{(i)} \in \Rmn{n_s}{\ell_s}$ with $\ell_s \ll n_s$, and $\Ldir_j^{(i)} = [\ldir_1 \cdots \ldir_{\ell_d}]_j^{(i)} \in \Rmn{n_d}{\ell_d}$ with $\ell_d \ll n_d$ \cite{HabeChunHerr2012, AslanStuKilmer}. 
The numbers of samples, $\ell_s$ and 
$\ell_d$, must be somewhat larger
than the (expected) ranks of the
sampled blocks.

We use the Rademacher distribution to generate
$\Rdir_j^{(i)}$ and $\Ldir_j^{(i)}$; but other distributions can be used as well.
This approach significantly reduces the number of large linear system solves and adds only $(\ell_s+\ell_d)$ directions to the ROM basis per frequency and interpolation point as opposed to $(n_s+n_d)$. See Algorithm \ref{Alg:ROMrandsrcdet}.

\begin{algorithm}
\caption{} \label{Alg:ROMrandsrcdet}
Given parameter sample points $\pmb\pi_i$, frequency interpolation points $\omega_j$, and random matrices $\Ldir^{(i)}_j \in \Rmn{n_d}{\ell_d}$ and $\Rdir^{(i)}_j \in \Rmn{n_s}{\ell_s}$ for $i=1,\cdots, n_k$ and $j=1,\cdots, n_\o$,
\begin{itemize}
  \item[(1)] Generate candidate bases for $i=1,\cdots {n_k}$\\
 $ \MVrand^{(i)}=\left[\left( \frac{\imath\o_1}{\nu}\ME-\MA(\pmb\pi_i)\right)^{-1}\MB\Rdir^{(i)}_{1}, \cdots,  \left(\frac{\imath\o_{n_\o}}{\nu}\ME-\MA(\pmb\pi_i)\right)^{-1}\MB\Rdir^{(i)}_{n_\o}	\right]$\\
  $ \MWrand^{(i)}=\left[\left( \frac{\imath\o_1}{\nu}\ME-\MA(\pmb\pi_i)\right)^{-T}\MC^T\Ldir^{(i)}_{1}, \cdots,  \left(\frac{\imath\o_{n_\o}}{\nu}\ME-\MA(\pmb\pi_i)\right)^{-T}\MC^T\Ldir^{(i)}_{n_\o}			\right]$
\item[(2)] Let $\MVrand = [\MVrand^{(1)},\MVrand^{(2)},\cdots,\MVrand^{({n_k})}]$ and $\MWrand = [\MWrand^{(1)},\MWrand^{(2)},\cdots,\MWrand^{({n_k})}]$.
 \item [(3)]Use SVDs or rank-revealing QR factorizations of
 $[\MVrand^{(1)},\MVrand^{(2)},\cdots,\MVrand^{({n_k})}]$ and  $[\MWrand^{(1)},\MWrand^{(2)},\cdots,\MWrand^{({n_k})}]$
 to compute ROM bases $\mathbf{Q}_{\MVrand}$ and $\mathbf{Q}_{\MWrand} $,
 respectively. Then, set  $\MVrand_r \leftarrow [\mathbf{Q}_{\MVrand}, \ \mathbf{Q}_{\MWrand}]$
 (one-sided projection) to obtain the global basis $\MVrand_r\: (=\MWrand_r)$.
\end{itemize}
\end{algorithm}
Algorithm \ref{Alg:ROMrandsrcdet} computes an alternative global basis $\MVrand_r$, which can be used as in (\ref{Eq:TransferFunc})--(\ref{Eq:ROMeq})
to compute the reduced transfer function and the reduced order model for every parameter $\Vp$. This approach can easily be combined with the approach introduced in \cite{OConKilStuGug} to compute
only necessary extensions of the global basis. However, this idea is not exploited in this paper.

We will provide a theoretical justification for the low-rank structure of the candidate basis $\MV$ resulting from Algorithm \ref{Alg:ROMallsrcdet}
and a justification for Algorithm \ref{Alg:ROMrandsrcdet} in the next section. However, we first illustrate these concepts numerically.
We consider a 2D experiment with 32 sources and 32 detectors. We use 4 parameter sample points and 2 frequencies to construct $\MV$ using Algorithm \ref{Alg:ROMallsrcdet}. Figure \ref{Fig:SingValsV} shows the fast decay of the singular values of the candidate basis. We also give the cosines of the canonical angles between $\Ra{\MV_r}$, the global basis computed using all sources and detectors, and  $\Ra{\MVrand_r}$, the global basis computed using randomized sources and detectors, for the same experiment in Figure \ref{Fig:CanonicalAng}. Clearly, the global basis computed using Algorithm \ref{Alg:ROMrandsrcdet} is very close to the global basis computed using Algorithm \ref{Alg:ROMallsrcdet}. So, for this problem, the proposed model reduction basis
computed using random simultaneous sources and detectors spans
almost the same subspace as the
full interpolation basis (ignoring directions corresponding to tiny singular values). Our approach constructs a global basis $\MVrand_r$ such that $\Ra{\MVrand_r} \approx \Ra{\MV_r}$ 
and the reduced transfer function $\Psirand_r(\o;\Vp)$ still satisfies the interpolation conditions (\ref{Eq:ROMhermite})
for $\pmb{\Psi}(\o;\Vp)$ at the interpolation points approximately (and typically quite accurately). Algorithm \ref{Alg:ROMrandsrcdet} greatly reduces the cost of building ROMs.
In the next section, we provide some further theoretical motivation for our approach.

\begin{figure}[t]
\begin{minipage}{.48\textwidth}
  \centering
 \includegraphics[width=1\linewidth]{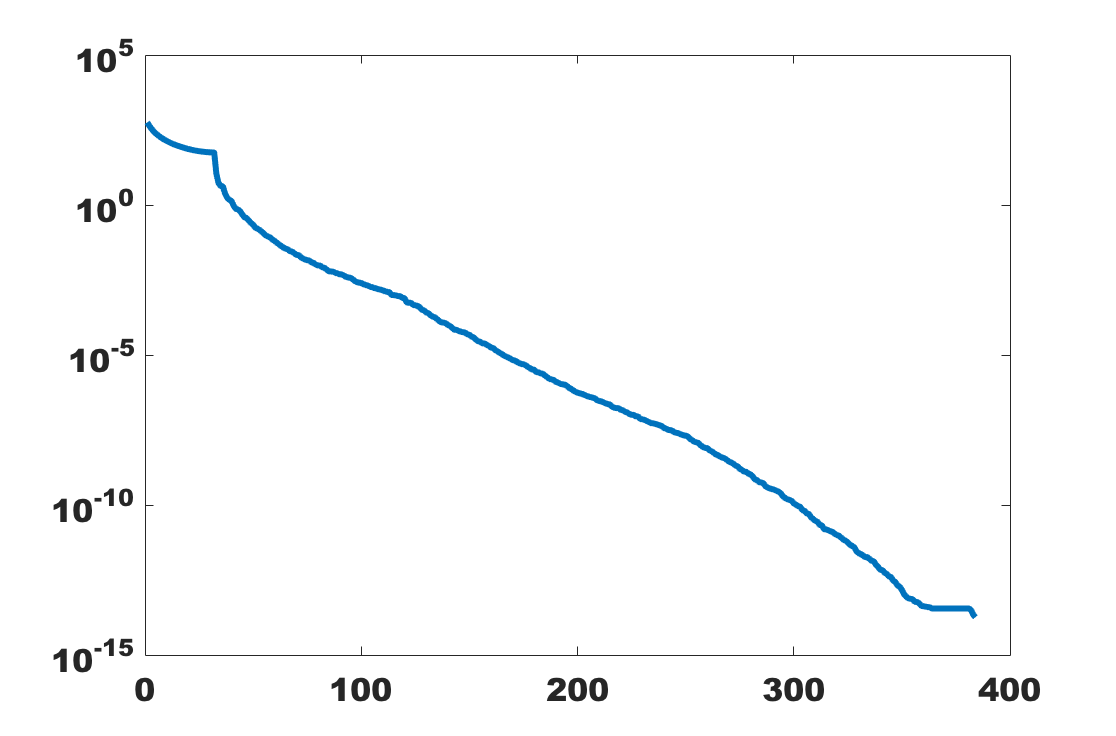}
   \caption{Singular values of the candidate basis $\MV$ before computing global basis.}
   \label{Fig:SingValsV}
\end{minipage}\:\:\:
\begin{minipage}{.48\textwidth}
  \centering
  \includegraphics[width=1\linewidth]{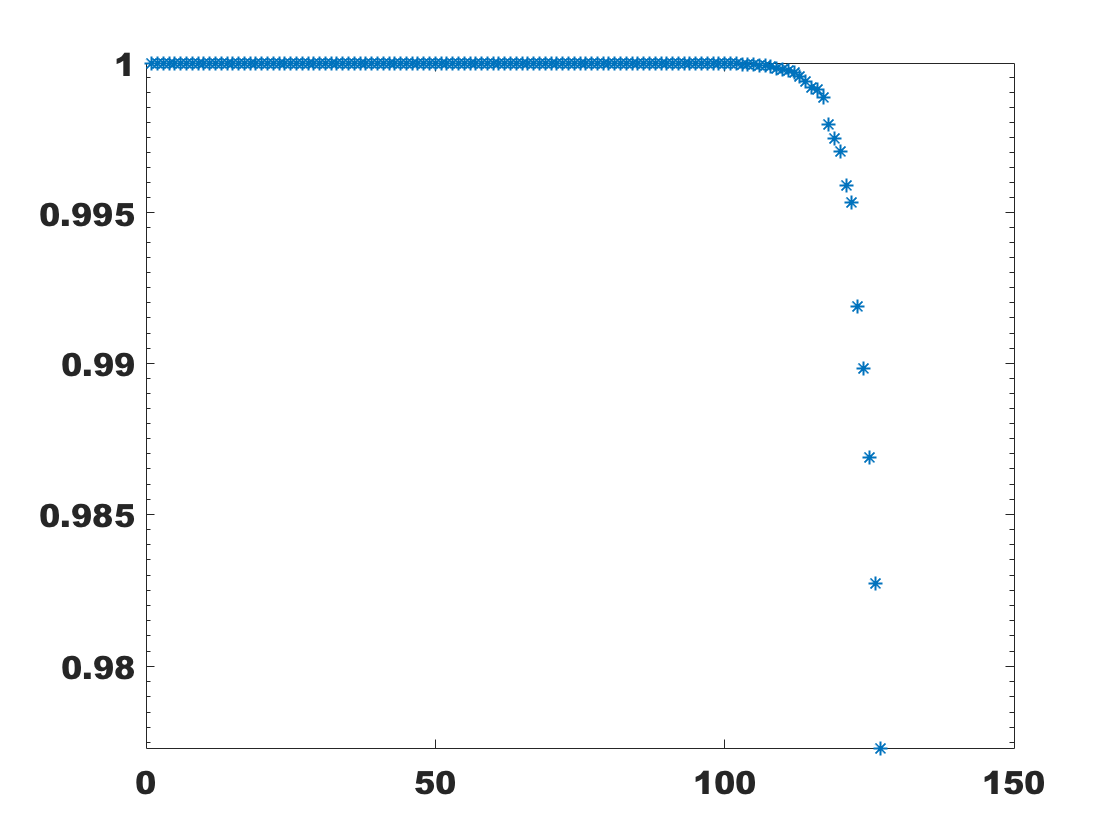}
  \caption{The cosines of the canonical angles between $\Ra{\MV_r}$ and $\Ra{\MVrand_r}$.}
  \label{Fig:CanonicalAng}
\end{minipage}
% \vspace{-0.2in}
\end{figure}

Before we start developing the theoretical justification for Algorithm \ref{Alg:ROMrandsrcdet}, we note that, in the setting of interpolatory model reduction,
our approach may also be interpreted as picking random tangential interpolation directions. Tangential interpolation refers to interpolating a matrix-valued function only along selected directions \cite{AntBG20}, as opposed to enforcing full matrix interpolation as we would like to enforce in DOT as shown in \eqref{Eq:ROMhermite}.
By rephrasing the tangential interpolation theory developed for parametric systems \cite{BauBBG09}, we can state that the reduced parametric model using the global basis $\MVrand_r$ from Algorithm \ref{Alg:ROMrandsrcdet} satisfies the following tangential interpolation conditions for $j=1,\cdots, n_\o$ and $i=1,\cdots, n_k$:
\begin{align}
  (\Ldir^{(i)}_j)^T \pmb{\Psi}(\omega_j;\pmb{\pi}_i) = (\Ldir^{(i)}_j)^T \Psirand_r(\omega_j;\pmb{\pi}_i),   
    &\:\:\:   \:\:\:
  \pmb{\Psi}_r(\omega_j;\pmb{\pi}_i)\Rdir^{(i)}_j = \Psirand_r(\omega_j;\pmb{\pi}_i)\Rdir^{(i)}_j,  \nonumber\\ 
  \nabla_{\Vp}\:((\Ldir^{(i)}_j)^T \pmb{\Psi}(\omega_j;\pmb{\pi}_i)\Rdir) & =  \nabla_{\Vp}\: ((\Ldir^{(i)}_j)^T \Psirand_r(\omega_j;\pmb{\pi}_i)\Rdir^{(i)}_j),~\mbox{and}      \label{Eq: OptCond2}   \\\nonumber
  (\Ldir^{(i)}_j)^T \pmb{\Psi}^\prime(\omega_j;\pmb{\pi}_i)\Rdir& = (\Ldir^{(i)}_j)^T \Psirand_r^\prime(\omega_j;\pmb{\pi}_i)\Rdir^{(i)}_j.
\end{align}
The difference is clear: As opposed to interpolating the full-matrix valued functions such as $\pmb{\Psi}(\omega_j;\pmb{\pi}_i)$, we only interpolate them along tangential directions selected from the Rademacher distribution. However, despite this theoretical link to tangential interpolation, our emphasis in this paper is to approximate the full matrix interpolation (all directions) property as in \eqref{Eq:ROMhermite} using the ROM projection basis obtained via Algorithm 
\ref{Alg:ROMrandsrcdet}. This is crucial in the DOT setting since matching the cost function and parameter gradient require full-matrix interpolation; not a tangential one. Therefore, our goal is to justify analytically and numerically that, in the DOT setting, the efficiently computed reduced basis via randomized sources and detectors, approximates the full global basis accurately and thus provides nearly full matrix interpolation. Exploring potential links with tangential interpolation is an interesting 
direction for future research, which might help identifying how to choose these directions optimally. The recent work \cite{hund2018H2L2} might provide a link in this direction.

\section{Combining Randomization for Efficient  ROM}\label{sec:TheorJust}
In the previous section, we showed for a test problem that the model reduction space computed by Algorithm \ref{Alg:ROMrandsrcdet} is very close to the model reduction space computed by Algorithm \ref{Alg:ROMallsrcdet}. In the following  subsections, we provide theoretical motivation for the low-rank structure of the candidate global basis computed using Algorithm \ref{Alg:ROMallsrcdet}, thus justifying the use of random sampling in Algorithm \ref{Alg:ROMrandsrcdet}. 

To make the proofs and presentation more concise, we first rewrite the transfer function in terms of a symmetric positive definite (SPD) coefficient matrix as in \cite{KilStu2006,OConKilStuGug}. Then, we analyze the relative changes in the candidate bases computed by Algorithm \ref{Alg:ROMallsrcdet} for consecutive parameters to motivate Algorithm \ref{Alg:ROMrandsrcdet}. 

\subsection{Rewriting the Transfer Function}\label{Sec:RewritingH}
For completeness, we show that the full order transfer function $\pmb{\Psi}(\o;\Vp)$ in (\ref{Eq:FOMtransferFun}) can be \emph{equivalently} written in terms of an SPD coefficient matrix. We follow the discussion in \cite{KilStu2006,OConKilStuGug} based
on the same discretization scheme. Let the matrix $\left(\frac{\imath \omega_j}{\nu} \ME +\MA(\Vp)\right)$ have the following block structure for the two-dimensional case,
\begin{equation}\label{Eq: BlockA}
\begin{bmatrix}
\MG & \MD_1\\
\MD_2 & \MF(\Vp)+ \frac{ \imath\omega h^2}{\nu} \MI
\end{bmatrix},
\end{equation}
where $\MG$ is an invertible diagonal matrix, $\MD_1$ has at most one nonzero per row (these occur only in the first $2N_x$ and the last $2N_x$ columns) and $\MD_2$ has the same sparsity pattern as $\MD^T_1$ (but with different entries). In addition, we have ordered the $N_xN_y$ unknowns such that the unknowns on the top and the bottom boundary appear first, followed by lexicographical ordering of the remaining unknowns. In addition, the matrices $\MC$ and $\MB$ contain selected columns from an  $N_xN_y\times N_xN_y$  identity matrix, and $\MB$ is scaled by $\displaystyle \frac{1}{h^2}$. Hence, the matrices $\MC$ and $\MB$ 
have the partitioning
\begin{equation}\label{Eq:MatrixC1B1 }
\begin{bmatrix}
\MC_1 \:\:
\textbf{0}
\end{bmatrix}, \begin{bmatrix}
\textbf{0}\\
\MB_1
\end{bmatrix},
\end{equation}
conforming with (\ref{Eq: BlockA}).

Next, we rewrite the differential-algebraic system (\ref{Eq:ModelDynSys}) with $\wh{\ME} =\frac{1}{\nu} \ME$
\begin{equation}\label{Eq:DynamSys}
\displaystyle \wh{\ME}\dot{\Vy}(t; \Vp)= -\MA(\Vp) \Vy(t; \Vp) + \MB \, \Vu(t)\:\:\: \text{with}\:\:\: \Vm(t; \Vp)  = \MC \, \Vy(t; \Vp),
\end{equation}
where $\Vy$ denotes the discretized photon flux, and $\Vm$ is the vector of detector outputs.
Partitioning the matrix $\wh{\ME}$ conforming with (\ref{Eq: BlockA}) gives
\begin{equation}\label{Eq:BlockE}
\wh{\ME}=\begin{bmatrix}
\pmb{0} &\pmb{0} \\
\pmb{0}  & \frac{1}{\nu} \MI
\end{bmatrix}.
\end{equation}
Substituting (\ref{Eq: BlockA}), (\ref{Eq:MatrixC1B1 }), and (\ref{Eq:BlockE}) into (\ref{Eq:DynamSys}) gives
\begin{equation}\label{Eq:SymDynSysDer}
\begin{bmatrix}
\pmb{0} &\pmb{0} \\
\pmb{0}  & \frac{1}{\nu} \MI
\end{bmatrix} \begin{bmatrix}
\dot{\Vy}_1  \\
\dot{\Vy}_2
\end{bmatrix}=-\begin{bmatrix}
\MG & \MD_1\\
\MD_2 & \MF(\Vp)
\end{bmatrix} \begin{bmatrix}
\Vy_1  \\
\Vy_2
\end{bmatrix}+ \begin{bmatrix}
\textbf{0}\\
\MB_1
\end{bmatrix}\Vu,
\end{equation}
and we obtain
\begin{align}
\pmb{0}&=-\MG \Vy_1-\MD_1\Vy_2, \label{Eq:y1}\\
\frac{1}{\nu} \dot{\Vy}_2&=-\MD_2 \Vy_1 - \MF(\Vp) \Vy_2+ \MB_1 \Vu. \label{Eq:y2}
\end{align}
Substituting $\Vy_1=-\MG^{-1} \MD_1 \Vy_2$ from (\ref{Eq:y1}) into (\ref{Eq:y2}) yields
\begin{equation}\label{Eq:Sym_y_2}
\frac{1}{\nu} \dot{\Vy}_2=-\left[\MF(\Vp)- \MD_2 \MG^{-1} \MD_1\right]\Vy_2+\MB_1 \Vu=\wt{\MA}(\Vp)\Vy_2+\wt{\MB}\Vu.
\end{equation}
It is straightforward to show that
\begin{equation}\label{Eq:SymVecOutput}
\Vm(t; \Vp)  = \MC\, \Vy(t; \Vp)= -\MC_1 \MG^{-1}\MD_1 \Vy_2=\wt{\MC} \Vy_2.
\end{equation}
Next, we define the transfer function, after eliminating the singular matrix $\ME$, as
\begin{equation} \label{eq:odetf}
\wh{\pmb{\Psi}}(\o;\Vp)=\wt{\MC}\left( \frac{\imath \omega}{\nu} \MI+ \wt{\MA}(\Vp)\right)^{-1} \wt{\MB}.
\end{equation}
Hence, from a system theoretic perspective, we converted a system of differential-algebraic equations (DAE) to a system of ordinary differential equations (ODE). For details on interpolatory model order reduction for DAEs, we refer the reader to \cite{GugStyWyatt2013}. In the remainder, we consider this system of ODEs and its transfer function (\ref{eq:odetf}).

\subsection{Perturbations in the Candidate Basis}\label{Sec:GlobalBasis}

In this section, we show that the candidate model reduction basis
constructed by Algorithm \ref{Alg:ROMallsrcdet} is low rank if the sample points $\pmb \pi_i$ correspond to small anomalies. Hence, the coefficient matrix $\MAsym(\Vp)$
is a small perturbation of the discretized Laplacian, which is the coefficient matrix for the case of no anomaly, i.e., $\MAsym(\V0)$. 
In this case, we can bound the norm of the difference between any two solutions corresponding to small anomalies by twice the norm of the difference between a solution for a small anomaly and the solution for the Laplacian (no anomaly).
In this section, $\|.\|$ denotes the Frobenius norm.

We anticipate that due to the low-rank structure of the candidate basis computed using Algorithm \ref{Alg:ROMallsrcdet}, the model reduction basis 
from  Algorithm \ref{Alg:ROMrandsrcdet}
using randomized sources and detectors essentially captures the same space. We justify this below in a slightly modified setting.

In Section \ref{Sec:RewritingH}, we showed that the transfer function can be equivalently written as
\begin{equation}\label{eq:newH}
\wh{\pmb{\Psi}}(\omega,\Vp)= \wt{\MC} \left(  \frac{ \imath\omega}{\nu} \MI +\wt{\MA} (\Vp)\right)^{-1} \wt{\MB},
\end{equation}
where $\MAsym(\Vp)$ is SPD and (by assumption) a small perturbation of the discretized Laplacian. Next, we show that relative changes in the candidate basis under small changes in the parameter vector, computed using Algorithm \ref{Alg:ROMallsrcdet}, remain small. For simplicity, we assume $\o=0$ in the following derivations.

Since we use a finite difference discretization, we note that, for changes in $\Vp$, the changes in $\MAsym(\Vp)$ are small relative to the magnitude of the matrix coefficients and occur only on the diagonal; see \cite{StuGuKilChatBeatOCon}. Since the optimization constrains the updates in $\Vp$, in each step, to be relatively small, after a few iterations, the changes in $\MAsym(\Vp)$ are also highly localized. Hence, the changes in the discretized matrix are limited to a few diagonal coefficients that
are also small, i.e.,
\begin{equation}\label{Eq:DeltaAperturb}
\wh{\Delta \MA} =\MAsym(\pmb{\pi}_{i+1})-\MAsym(\pmb{\pi}_i)= \wt{\Delta \MA}^{(i+1)} - \wt{\Delta \MA}^{(i)},
\end{equation}
where $\wt{\Delta \MA}^{(i+1)} = \MAsym(\pmb{\pi}_{i+1})-\MAsym(\mathbf{0})$ and $\wt{\Delta \MA}^{(i)} = \MAsym(\pmb{\pi}_i)-\MAsym(\mathbf{0})$. Here, $\MAsym(\mathbf{0})$ corresponds to the discretized Laplacian (no anomaly). In addition, the nonzero diagonal coefficients in $\wt{\Delta \MA}^{(i+1)}$ and
$\wt{\Delta \MA}^{(i)}$ are of magnitude $O(h^2)$.
Therefore, $\wt{\Delta \MA}^{(i)}$ satisfies
\begin{equation}\label{Eq:DeltaA_UY}
  \wt{\Delta \MA}^{(i)} = \MUtil^{(i)} \Xi^{(i)} (\MUtil^{(i)})^T,
\end{equation}
where $\Xi^{(i)}$ is a $k\times k$ diagonal matrix with entries of magnitude $O(h^2)$ and $k$ is a modest number. $\MUtil^{(i)}$ is given by
\begin{equation}\label{Eq:Utildef}
  \MUtil^{(i)} =\left[\Ve_{j_1}\:\: \Ve_{j_2}\:\: \cdots \:\: \Ve_{j_k} \right],
\end{equation}
where $\Ve_{j}$ denotes the $j^{th}$ Cartesian basis vector and the indices $j_k$ correspond to pixels where absorption has changed. Similarly,
$\wt{\Delta \MA}^{(i+1)}$ satisfies
\begin{equation}\label{Eq:DeltaA_UY2}
  \wt{\Delta \MA}^{(i+1)} = \MUtil^{(i+1)} \Xi^{(i+1)}(\wt{\MU}^{(i+1)})^T.
\end{equation}
Next, we theoretically justify that the updates in the candidate basis are indeed low rank based on the special structure of
$\wt{\Delta \MA}^{(i)}$ and $\wt{\Delta \MA}^{(i+1)}$ and
the magnitude of the diagonal coefficients, which remains small as $\Vp$ changes.

\begin{lemma}\label{Lemma:DeltaVSher} Let $\Delta{\MV}=\MV^{(i+1)}-\MV^{(i)}$ denote the change in a candidate basis, where $\MV^{(i)}$ and $\MV^{(i+1)}$ are defined as in Algorithm \ref{Alg:ROMallsrcdet} for $\o=0$.
Let $\MAsym(\pmb{\pi}_i)$ and $\MAsym(\pmb{\pi}_{i+1})$ correspond to the discretized
matrices for parameter vectors $\pmb{\pi}_i$ and $\pmb{\pi}_{i+1}$, and define $\MDeltasym=\MAsym(\pmb{\pi}_i)-\MAsym(\V0)$ as in (\ref{Eq:DeltaA_UY}), where
$\MAsym(\mathbf{0})$  is the discretized Laplacian. Then,

\vspace{-2ex}
{\small\begin{align}\label{Eq:DeltaVboundLemma}
\frac{\|\Delta\MV\|}{\|\MV^{(0)}\|} &\leq
   \left\| \MAsym(\mathbf{0})^{-1} \MUtil^{(i)}  \right\| \:
   \left\| \left( \MI_k+ \Xi^{(i)} \MYtil  \MAsym(\mathbf{0})^{-1} \MUtil^{(i)} \right)^{-1} \right\| \: \left\| \Xi^{(i)} \MYtil \right\|        \\
   &+
   \left\| \MAsym(\mathbf{0})^{-1} \MUtil^{(i+1)}  \right\| \:
   \left\| \left( \MI_k+ \Xi^{(i+1)} (\wt{\MU}^{(i+1)})^T  \MAsym(\mathbf{0})^{-1}
   \MUtil^{(i+1)} \right)^{-1} \right\| \:
   \left \| \Xi^{(i+1)} ( \wt{\MU}^{(i+1)} )^T \right\|,
   \nonumber
\end{align}}
where $\MV^{(0)}=\MAsym(\Vo)^{-1}\MBt$ and $\MI_k$ is the identity matrix of size $k\times k$.
\end{lemma}
\begin{proof} From Algorithm \ref{Alg:ROMallsrcdet}, we get
\begin{align}\label{Eq:DeltaVdef}
\Delta\MV=\MV^{(i+1)}-\MV^{(i)}=[\MV^{(i+1)}-\MV^{(0)}]- [\MV^{(i)}-\MV^{(0)}].
\end{align}
Using
\begin{align*}
\MAsym(\pmb{\pi}_i)&=\MAsym(\V0)(\MI_n+\MAsym(\V0)^{-1}\wt{\Delta \MA}^{(i)}),
\end{align*}
we obtain
\begin{align}\label{Eq:AinvB_IDeltaA}
\MAsym(\pmb{\pi}_i)^{-1}\MBt &=(\MI_n+\MAsym(\V0)^{-1}\MDeltasym)^{-1}\MAsym(\V0)^{-1}\MBt,
\end{align}
where $\MAsym(\pmb{\pi}_i)$ is invertible for any $\Vp$. Substituting (\ref{Eq:DeltaA_UY}) into (\ref{Eq:AinvB_IDeltaA}), using the Sherman-Morrison-Woodbury formula, and
$\MV^{(0)}=\MAsym(\Vo)^{-1}\MBt$, we get
{\small\begin{equation}\label{Eq:V2V1proof2}
\MAsym(\pmb{\pi}_i)^{-1}\MBt =
  \left[ \MI_n-\MAsym(\V0)^{-1}\MUtil^{(i)}
  \left( \MI_k+\Xi^{(i)} \MYtil\MAsym(\V0)^{-1}\MUtil^{(i)} \right)^{-1}
  \Xi^{(i)} \MYtil \right]  \MV^{(0)}.
\end{equation}}
Similarly, we have
{\small \begin{equation}\label{Eq:V2V0proof2}
\MAsym(\pmb{\pi}_{i+1})^{-1}\MBt=
  \left[ \MI_n-\MAsym(\V0)^{-1}\MUtil^{(i+1)}
  \left( \MI_k+\Xi^{(i+1)}
     ( \wt{\MU}^{(i+1)} )^T \MAsym(\V0)^{-1} \MUtil^{(i+1)} \right)^{-1}
     \Xi^{(i+1)} ( \wt{\MU}^{(i+1)} )^T \right]  \MV^{(0)}.
\end{equation}}
Using (\ref{Eq:DeltaVdef}), \eqref{Eq:V2V1proof2} and \eqref{Eq:V2V0proof2}, we obtain
\begin{align}
\Delta\MV
  &=
  \left[ \MAsym(\V0)^{-1}\MUtil^{(i)}
    \left( \MI_k+\Xi^{(i)} \MYtil\MAsym(\V0)^{-1}\MUtil^{(i)} \right)^{-1}\Xi^{(i)} \MYtil
    \right.
    \nonumber \\
  &-
  \left.
  \MAsym(\V0)^{-1}\MUtil^{(i+1)}
      \left( \MI_k+\Xi^{(i+1)} ( \wt{\MU}^{(i+1)} )^T \MAsym(\V0)^{-1}\MUtil^{(i+1)} \right)^{-1} \Xi^{(i+1)} ( \wt{\MU}^{(i+1)} )^T
      \right] \MV^{(0)},
\label{Eq:DeltaVlowrnk}
\end{align}
where $\MAsym(\V0)^{-1}\MUtil^{(i)} \in \Rmn{n}{k}$ and $\MAsym(\V0)^{-1}\MUtil^{(i+1)} \in \Rmn{n}{k}$. Taking norms and using the submultiplicative property in \eqref{Eq:DeltaVlowrnk} yield

\vspace{-2ex}
{\small \begin{align*}
\|\Delta\MV\| &\leq
    \left[  \| \MAsym(\mathbf{0})^{-1} \MUtil^{(i)}
    \left( \MI_k+ \Xi^{(i)} \MYtil  \MAsym(\mathbf{0})^{-1} \MUtil^{(i)} \right)^{-1} \Xi^{(i)} \MYtil \| \right.
    \\
 &+
   \left.
   \| \MAsym(\mathbf{0})^{-1} \MUtil^{(i+1)}
   \left( \MI_k+ \Xi^{(i+1)} ( \wt{\MU}^{(i+1)} )^T
   \MAsym(\mathbf{0})^{-1} \MUtil^{(i+1)} \right)^{-1} \Xi^{(i+1)}
   ( \wt{\MU}^{(i+1)} )^T \| \right] \| \MV^{(0)}\|.
\end{align*}}
Then, the relative change in $\MV^{(i)}$ is bounded by
\begin{align*}
\frac{\| \Delta\MV \|}{\| \MV^{(0)} \|} &\leq \| \MAsym(\mathbf{0})^{-1}
    \MUtil^{(i)}(\MI_k+ \Xi^{(i)} \MYtil  \MAsym(\mathbf{0})^{-1} \MUtil^{(i)} )^{-1} \Xi^{(i)} \MYtil \|
\nonumber \\
   &+ \| \MAsym(\mathbf{0})^{-1} \MUtil^{(i+1)}
   ( \MI_k+ \Xi^{(i+1)} ( \wt{\MU}^{(i+1)} )^T \MAsym(\mathbf{0})^{-1} \MUtil^{(i+1)} )^{-1} \Xi^{(i+1)} ( \wt{\MU}^{(i+1)} )^T \| ,
\end{align*}
which gives the desired inequality
\begin{align}\label{Eq:DeltaVboundprf}
\frac{\| \Delta\MV \|}{\| \MV^{(0)} \|}
    &\leq
    \| \MAsym(\mathbf{0})^{-1} \MUtil^{(i)}  \| \:
    \| (\MI_k+ \Xi^{(i)} \MYtil  \MAsym(\mathbf{0})^{-1} \MUtil^{(i)} )^{-1} \| \:
    \| \Xi^{(i)} \MYtil \|
\nonumber \\
    &+
    \| \MAsym(\mathbf{0})^{-1} \MUtil^{(i+1)}  \| \:
    \| ( \MI_k+ \Xi^{(i+1)} ( \wt{\MU}^{i+1} )^T \MAsym(\mathbf{0})^{-1}
        \MUtil^{(i+1)} )^{-1} \| \:
    \| \Xi^{(i+1)} ( \wt{\MU}^{(i+1)} )^T \|.
\nonumber \\
\end{align}
\end{proof}

\begin{remark} Equation (\ref{Eq:DeltaVlowrnk}) shows that the changes in the candidate basis lie in $\Ra{\: \MAsym(\mathbf{0})^{-1} [\MUtil^{(i)} \:\:\:\: \MUtil^{(i+1)}] \:}$. Hence, the updates in the global basis are indeed low rank.
\end{remark}

{Next, we give a bound on (\ref{Eq:DeltaVboundLemma})}.
\begin{lemma} Let $\Delta \MV=\MV^{(i+1)}-\MV^{(i)}$ denote the change in a candidate basis, where $\MV^{(i+1)}$ and $\MV^{(i)}$ are defined as in Algorithm \ref{Alg:ROMallsrcdet}. Then,
 \begin{equation}\label{Eq:DeltaVBoundh}
\displaystyle   \frac{\| \Delta \MV \| }{\| \MV^{(0)}\|} \leq 2C h,
 \end{equation}
 where $C$ is a constant and $h$ is the mesh width.
\end{lemma}

\begin{proof} To bound $\displaystyle \frac{\| \Delta\MV \|}{\| \MV^{(0)} \|} $, we examine the right-hand side of (\ref{Eq:DeltaVboundprf}). We start with bounding
$\| \MAsym(\V0)^{-1} \MUtil^{(i)} \|$, where  $\MUtil^{(i)}$ is defined in (\ref{Eq:Utildef}). Since $k$ is a modest number, and we have
\begin{equation}
  \| \MAsym(\pmb{\pi}_i)^{-1} \MUtil^{(i)} \| \leq
  k_i \max_{e_j} \| \MAsym(\V0)^{-1} \Ve_j \| ,
\end{equation}
we focus on $||\MAsym(\V0)^{-1} \Ve_j||$. Since $\MAsym(\V0)$ is SPD, we express $\MAsym(\V0)$ in terms of its eigendecomposition,
\begin{equation}\label{Eq:EigdecompAtilda}
\MAsym(\V0)=\EigvecsUnit \pmb{\Lambda}\EigvecsUnit^T.
\end{equation}
In the following, we need the eigenvalues and eigenvectors of the 2D discretized Laplacian, see \cite{ZhilinFDbook2018}. We assume the mesh width $h$ is the same in the $x$ and $y$ directions. Let $k_x$ and $k_y$ be the wave numbers in the $x$ and $y$ directions. Labeling the eigenvectors by the wave number, the components of the normalized $(k_x, k_y)$ eigenvector are given by
 \begin{equation}\label{Eq:Lapl_eigvec_unitary}
\eigvecUnitD=\nu^{k_x k_y}\sin(i_x k_x \pi h) \sin(i_y k_y \pi h),\:\: i_x,\:i_y=1,\cdots,K-1,
\end{equation}
where the scaling factor $\nu^{k_x k_y} = 2h$, see \cite[pg. 400]{Gockenbach_PDEbook}, with the corresponding eigenvalues
\begin{equation}\label{Eq:Lapl_eigs}
\lambda^{k_x,k_y}=4\left[\sin^2\left(\frac{k_x \pi h}{2}\right)+\sin^2\left(\frac{k_y \pi h}{2}\right)  \right].
\end{equation}
Using the eigendecomposition of $\MAsym(\V0)$ and the fact that $\EigvecsUnit$ is orthogonal, we obtain
\begin{align}\label{Eq:Ainv_ej}
\| \MAsym(\V0)^{-1} \Ve_j \|=\|\EigvecsUnit \pmb{\Lambda}^{-1}\EigvecsUnit^T\Ve_j\|=\|\pmb{\Lambda}^{-1}\EigvecsUnit^T\Ve_j\|.
\end{align}
Then, using (\ref{Eq:Lapl_eigvec_unitary}), (\ref{Eq:Lapl_eigs}) and (\ref{Eq:Ainv_ej}) gives
\begin{align}\label{Eq:Vnorm_p1}
\displaystyle ||\pmb{\Lambda}^{-1}\EigvecsUnit^T\Ve_j ||^2 &= \sum\limits_{k_y=1}^{K-1}\sum\limits_{k_x=1}^{K-1} \left[\frac{\nu^{k_x,k_y} \sin(j k_x \pi h)\sin(j k_y \pi h)}{4\left[\sin^2\left(\frac{k_x \pi h}{2}\right)+\sin^2\left(\frac{k_y \pi h}{2}\right)   \right]}\right]^2 \nonumber \\
&= h^2 \sum\limits_{k_y=1}^{K-1}\sum\limits_{k_x=1}^{K-1} \left[\frac{ \sin(j k_x \pi h)\sin(j k_y \pi h)}{\sin^2\left(\frac{k_x \pi h}{2}\right)+\sin^2\left(\frac{k_y \pi h}{2}\right) }\right]^2 \nonumber\\
&\leq h^2 \sum\limits_{k_y=1}^{K-1}\sum\limits_{k_x=1}^{K-1} \left[\frac{\frac{\pi^2}{4}}{\frac{\pi^2}{4} \left(\sin^2\left(\frac{k_x \pi h}{2}\right)+\sin^2\left(\frac{k_y \pi h}{2}\right) \right) }\right]^2.
\end{align}
Since $\displaystyle \beta^2\leq \frac{\pi^2}{4} \sin^2(\beta)$ for $\displaystyle 0 \leq \beta  \leq \frac{\pi}{2}$, we get
\begin{align}\label{Eq:Vnorm_p2}
||\pmb{\Lambda}^{-1}\EigvecsUnit^T\Ve_j ||^2 &\leq h^2 \sum\limits_{k_y=1}^{K-1}\sum\limits_{k_x=1}^{K-1} \left[\frac{\frac{\pi^2}{4}}{\frac{\pi^2 h^2}{4} \left(k_x^2+k_y^2\right) }\right]^2 \nonumber\\
&= \frac{1}{h^2} \sum\limits_{k_y=1}^{K-1}\sum\limits_{k_x=1}^{K-1} \left[\frac{1}{ \left(k_x^2+k_y^2\right) }\right]^2.
\end{align}
Since the double sum term is $O(1)$, we obtain the bound
\begin{align}\label{Eq:DeltaVbound_part1}
\| \pmb{\Lambda}^{-1}\EigvecsUnit^T\Ve_j  \| \leq  c_1 \frac{1}{h},
\end{align}
for a modest constant $c_1$.
Since $\| \Xi^{(i)} \MYtil \MAsym(\V0)^{-1} \MUtil^{(i)} \|$ is $O(h)$, it is now straightforward to show that
\begin{equation}\label{Eq:DeltaVbound_part2}
\| (\MI_k+\Xi^{(i)} \MYtil \MAsym(\V0)^{-1} \MUtil^{(i)})^{-1}\| \leq  c_2,
\end{equation}
where $c_2$ is a modest constant.
Finally, we have
\begin{equation}\label{Eq:DeltaVbound_part3}
 \| \Xi^{(i)} \MYtil \| \leq c_3 h^2
\end{equation}
from the definition of $\Xi^{(i)}$ (\ref{Eq:DeltaA_UY}). Combining (\ref{Eq:DeltaVbound_part1}), (\ref{Eq:DeltaVbound_part2}), and (\ref{Eq:DeltaVbound_part3}), we obtain the bound
\begin{equation}
\frac{\|\MV^{(i)} - \MV^{(0)}\|}{\|\MV^{(0)}\|} \leq C_i h,
\end{equation}
where $C_i=c_1 c_2 c_3$. Similarly, we can bound
\begin{equation}
  \frac{\| \MV^{(i+1)} - \MV^{(0)}\|}{\| \MV^{(0)} \|} \leq C_{i+1} h,
\end{equation}
where $C_{i+1}$ is a constant. Hence, we obtain the desired bound
\begin{equation}
\displaystyle \frac{\| \Delta \MV \| }{ \| \MV^{(0)} \|}
  \leq
  \frac{\|\MV^{(i)} - \MV^{(0)}\|}{\|\MV^{(0)}\|} +
  \frac{\|\MV^{(i+1)}- \MV^{(0)}\|}{\|\MV^{(0)}\|} \leq (C_{i+1} + C_{i}) h
      \leq 2Ch ,
\end{equation}
with $C = \max(C_{i+1},C_{i})$.
\end{proof}
 \begin{figure}
	\centering
		\includegraphics[ width=0.4\textwidth]{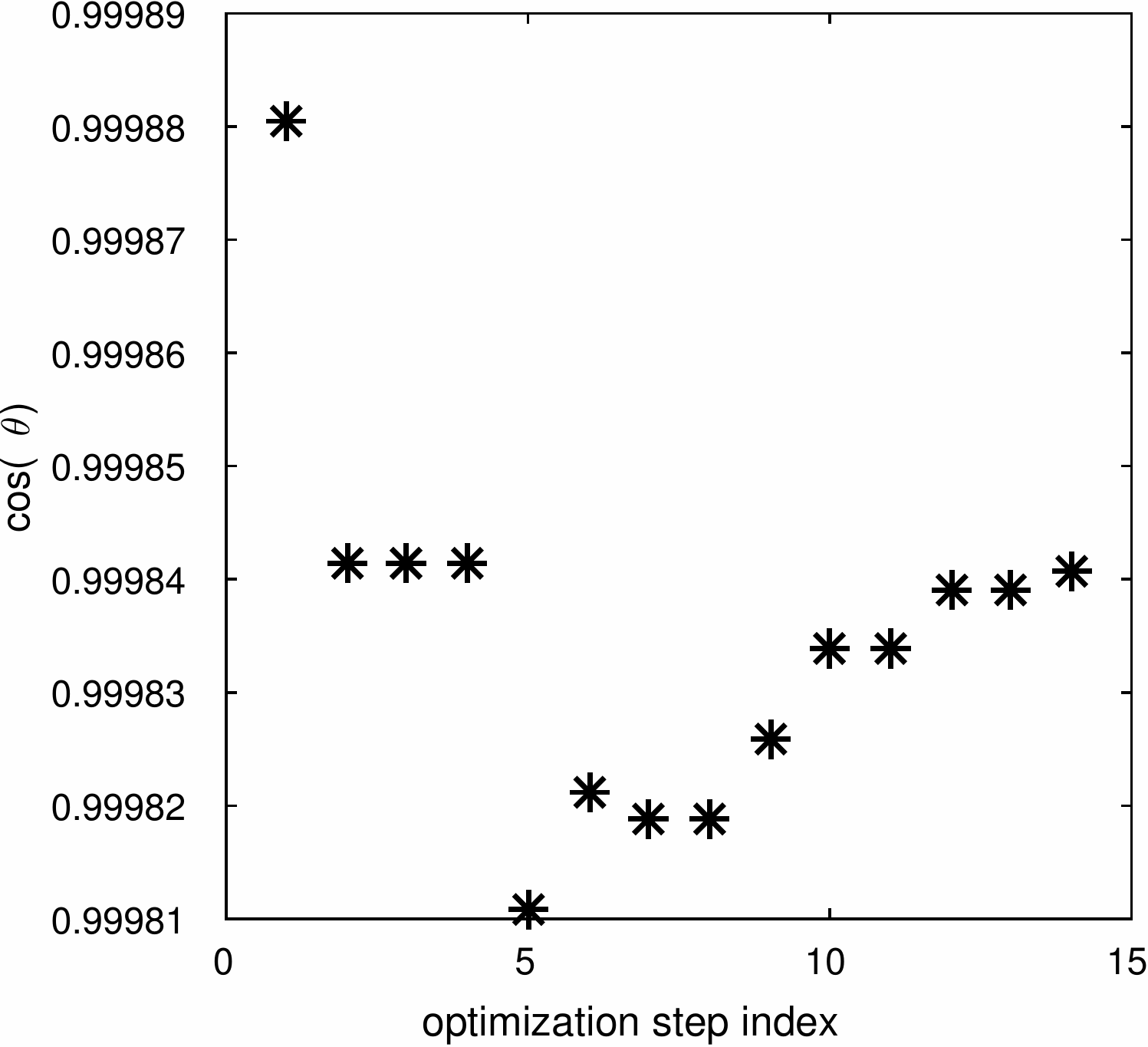}
\caption{Evolution of the subspace gap between the global basis, $\MVrand_r$, and new reduction spaces, $\MA(\Vp_s)^{-1} \MB$, over the course of the optimization. Example from the reconstruction of an anomaly on a $201\times 201$ mesh with 32 sources and 32 detectors. We use 25 basis functions and only the zero frequency.}
\label{Fig:CanAng_opt}
\end{figure}
We demonstrate this result in Figure~\ref{Fig:CanAng_opt} where we give the cosines of the canonical angles between the global basis, $\MVrand_r$, and new reduction spaces, $\MA(\Vp_s)^{-1} \MB$, over the course of the optimization. As can be seen, the angles remain small over the course of the optimization. Thus, even though we can only guarantee matching the cost function and the Jacobian at the sample points, the reduced parametric transfer function is expected to interpolate the full transfer function and the derivatives approximately, yet accurately enough, even for the other parameter points that the optimization algorithm visits during the inversion process.

\section{Numerical Results}\label{Sec:NumExp}
We present two proof-of-concept numerical experiments for 3D DOT inversion for absorption, using multiple frequencies. For each test case, we construct anomalies in the pixel basis, and  add a small normally distributed random heterogeneity to both the background and to the anomaly to make the medium inhomogeneous. This ensures a modest mismatch between the exact image and the representation used to reconstruct the image. We also add $\delta=0.1\%$ white noise to the simulated data in each experiment. We use PaLS \cite{Aghasi_etal11} and TREGS \cite{StuKil11c} to reconstruct the absorption images. The first ${n_k}$ iterations of the optimization using the exact objective function (FOM) give the parameter sample points.

To demonstrate the effectiveness of our proposed method, we use a subset of $\ell_s$ random simultaneous sources and $\ell_d$  random simultaneous detectors. As discussed in Section \ref{Sec:RandROM}, we use these randomized sources and detectors as sample vectors to efficiently approximate the global basis using Algorithm~\ref{Alg:ROMrandsrcdet}. We stop the optimization when the residual norm falls below $1.1$ times the noise level. For comparison, we run the FOM and the ROM using all sources and detectors as computed by Algorithm~\ref{Alg:ROMallsrcdet}. We use the same tolerance to remove small singular value components from candidate bases for both
Algorithms~\ref{Alg:ROMallsrcdet} and~\ref{Alg:ROMrandsrcdet}. 
For nonzero frequencies, following the common approach, after generating candidate bases in Algorithm~\ref{Alg:ROMallsrcdet}, we set the candidate bases $\MV \leftarrow [~\textsf{Re}(\MV)~~\textsf{Imag}(\MV)~]$ and $\MW \leftarrow [\textsf{Re}(\MW)~~\textsf{Imag}(\MW)]$ to  work in real arithmetic and retain the realness of the underlying dynamics. Similarly, for our approach, we set the candidate bases $\MVrand \leftarrow [~\textsf{Re}(\MVrand)~~\textsf{Imag}(\MVrand)~]$ and $\MWrand \leftarrow [\textsf{Re}(\MWrand)~~\textsf{Imag}(\MWrand)]$ in Algorithm~\ref{Alg:ROMrandsrcdet}. This also implies that the interpolation conditions (\ref{Eq:ROMhermite})
and (\ref{Eq: OptCond2}) hold also at the frequency point 
$-\omega_j$ for $j=1,\ldots,n_\omega$ (in addition to 
$\omega_j)$.

We present two 3D experiments using a $32\times 32\times 32$ mesh with 225 sources at the top and 225 detectors on the bottom. We use 27 CSRBFs to reconstruct both anomalies, leading to 135 parameters. The initial absorption image is given in Figure~\ref{Fig: Initial3D_RandROM} where 13 basis functions have a positive expansion coefficient (visible as high absorption regions) and 14 basis functions have a negative expansion coefficient (invisible). For our approach, we use 50 randomized sources and detectors to construct the ROM basis for each experiment, $\ell_s=\ell_d=50$. We vary the number of parameter sample points, $n_k$, and the number of frequencies, $n_{\o}$, in each experiment. 

\begin{figure}
	\centering
		\includegraphics[width=0.5\textwidth,]{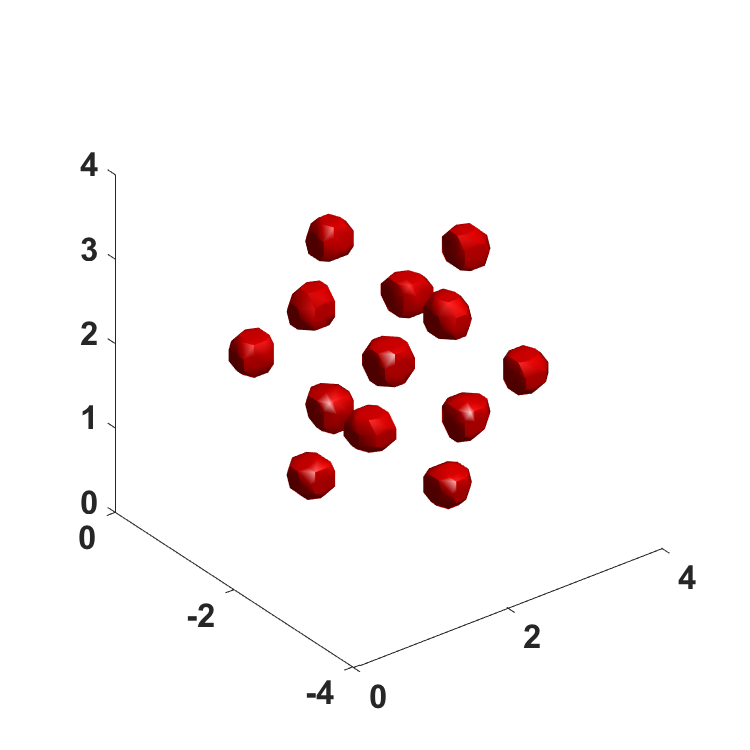}
	\caption{Initial configuration for the 3D experiments with 27 basis functions arranged in a $3\times 3 \times 3$ grid, where 13 basis functions have positive expansion factors (visible) and 14 basis functions have negative expansion factors (invisible).}
	\label{Fig: Initial3D_RandROM}
	\end{figure}
\textbf{Example 1.}  In this experiment, we use $n_k=4$ parameter sample points and $n_{\o}=3$ frequencies to construct the ROM basis. The true absorption image for Example 1 is given in Figure~\ref{Fig:3DExp_cup}a. Using the FOM, the optimization algorithm solves $18900$ linear systems of dimension $n = 32768$ to reconstruct the absorption image. The reconstruction result using the FOM is given in Figure~\ref{Fig:3DExp_cup}b. Constructing a ROM using all sources and detectors requires $5400$ large linear systems of dimension $n = 32768$. In addition, the optimization using the ROM requires $16875$ small linear systems of dimension $r = 1368$ to reconstruct the absorption image. The reconstruction result using the standard ROM with all sources and detectors is given in Figure~\ref{Fig:3DExp_cup}c. Constructing a ROM using our approach only requires $1200$ linear systems of dimension $n = 32768$. The optimization using the ROM with our approach requires $17550$ linear systems of dimension $r = 1294$ to reconstruct the absorption image. Our approach reduces the large solver cost by about \textit{a factor 16}, while obtaining similar quality reconstruction results; see Figure~\ref{Fig:3DExp_cup}d. In Table~\ref{Table:PDEsolves}, we give the number of linear systems to be solved for each method for comparison. In our experiments, the  reduced dimension, $r$, is larger
than one might usually expect in a reduced model. The need for this modest size $r$ stems from the fact that we have a high dimensional parameter space and the model to approximate has large input (source) and output (detector) dimensions. Thus, it is not surprising that
for a reduced model to have high-fidelity for these high-dimensional parameter, input, and output spaces, we need a modest size $r$. The reduced dimension $r$ could be cut in half if one chose not to retain the symmetry and   apply a two-sided reduction, i.e., $\MV_r \neq \MW_r$. However, we choose to preserve the underlying structure.
\begin{table}[]
\centering
\begin{tabular}{|c|c|c|c|c|c|}
\hline
    & FOM & \multicolumn{2}{c|}{\begin{tabular}[c]{@{}c@{}}ROM\\
    (all srcs/det)\end{tabular}} & \multicolumn{2}{c|}{\begin{tabular}[c]{@{}c@{}}ROM \\
    (rand. srcs/dets) \end{tabular}} \\
\hline
    & large & large & small & large & small \\
\hline
\begin{tabular}[c]{@{}c@{}}Exp. 1\\ ($n=32768$)\end{tabular} & $18900$ & $5400$          & \begin{tabular}[c]{@{}c@{}}16875\\ ($r=1368$)\end{tabular}          & $1200$            & \begin{tabular}[c]{@{}c@{}}17550\\ ($r=1294$)\end{tabular}          \\ \hline
\begin{tabular}[c]{@{}c@{}}Exp. 2\\ ($n=32768$)\end{tabular} & $47700$ & $5400$          & \begin{tabular}[c]{@{}c@{}}21600 \\ ($r=1228$)\end{tabular}         & $1200$            & \begin{tabular}[c]{@{}c@{}}27000\\ ($r=1184$)\end{tabular}           \\ \hline
\end{tabular} 
\vspace{1ex}
\caption{Number of large, $O(n)$, and small, $O(r)$, linear system solves
for each method for the 3D experiments.}
\label{Table:PDEsolves}
\end{table}

 \begin{figure}
	\centering
	   \begin{subfigure}{0.45\textwidth}
		\centering
		\includegraphics[ width=0.95\textwidth]{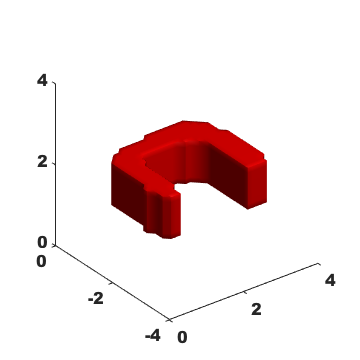}
		\subcaption{True shape of the anomaly.}
	\end{subfigure}\:\:\:
		\centering
	   \begin{subfigure}{0.45\textwidth}
		\centering
		\includegraphics[width=0.95\textwidth]{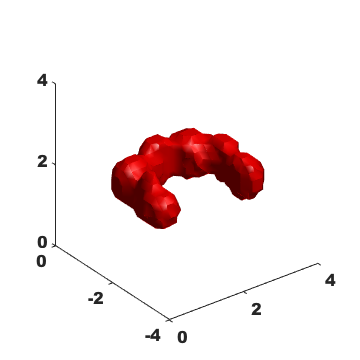}
		\subcaption{Reconstruction using the FOM.}
	\end{subfigure}
	\centering
		  \begin{subfigure}{0.45\textwidth}
		\centering
		\includegraphics[width=0.95\textwidth]{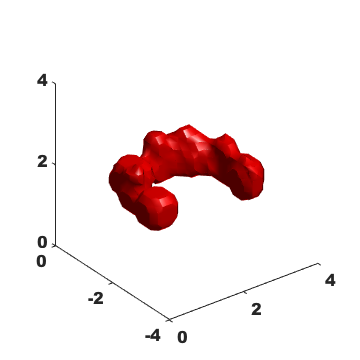}
		\subcaption{Reconstruction using the standard ROM with all sources and detectors.}
	\end{subfigure}\:\:\:
	\centering
	   \begin{subfigure}{0.45\textwidth}
		\centering
		\includegraphics[width=0.95\textwidth]{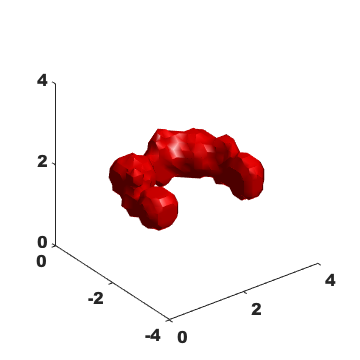}
		\subcaption{Reconstruction using the ROM with randomized sources and detectors, $\ell_s=50$.}
	\end{subfigure}
\caption{Results for Experiment 1. Reconstruction of a test anomaly on a $32\times 32 \times 32$ mesh with 225 sources and 225 detectors. We use 3 frequencies and 27 CSRBFs. To build the ROM, we use 4 parameter sample points.}
\label{Fig:3DExp_cup}
\end{figure}

\textbf{Example 2.} The setup is the same as in Example 1. However, for this experiment, we use $n_k=3$ parameter sample points and $n_{\o}=4$ frequencies to construct the ROM basis for a different anomaly, see Figure~\ref{Fig:3DExp}a. Using the FOM, the optimization algorithm solves $47700$ linear systems of dimension $n = 32768$ to reconstruct the absorption image. The reconstruction result using the FOM is given in Figure~\ref{Fig:3DExp}b. Constructing the ROM using all sources and detectors requires $5400$ linear systems of dimension $n = 32768$. The optimization using the ROM requires $21600$ linear systems of dimension $r = 1228$ to reconstruct the absorption image. The reconstruction result using the ROM constructed with all sources and detectors is given in Figure~\ref{Fig:3DExp}c. Constructing the ROM using our approach only requires $1200$ linear systems of dimension $n = 32768$. The optimization using the ROM with our approach requires $27000$ linear systems of dimension $r=1184$ to reconstruct the absorption image. Our approach reduces the large solver cost by about \textit{a factor 40}, while obtaining similar quality reconstruction results; see Figure~\ref{Fig:3DExp}d. The number of linear systems required for each method is given in Table~\ref{Table:PDEsolves}. For large problems with many sources and detectors and using multiple frequencies, we expect much larger gains.
\begin{figure}
	\centering
	   \begin{subfigure}{0.45\textwidth}
		\centering
		\includegraphics[ width=0.95\textwidth]{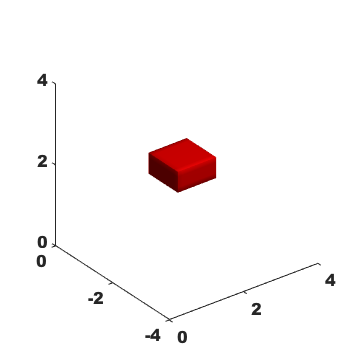}
		\subcaption{True shape of the anomaly.}
	\end{subfigure}\:\:\:
		\centering
	   \begin{subfigure}{0.45\textwidth}
		\centering
		\includegraphics[width=0.95\textwidth]{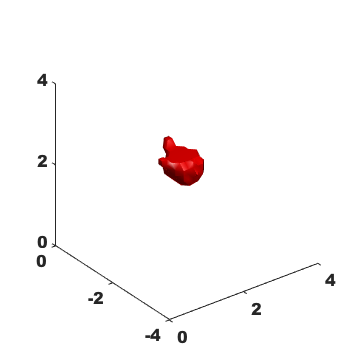}
		\subcaption{Reconstruction using the FOM.}
	\end{subfigure}
	\centering
		  \begin{subfigure}{0.45\textwidth}
		\centering
		\includegraphics[width=0.95\textwidth]{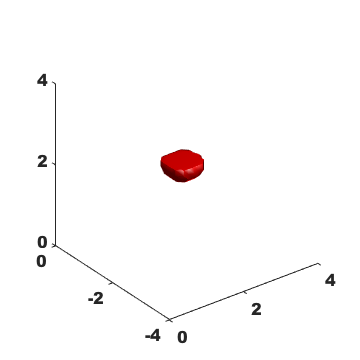}
		\subcaption{Reconstruction using the standard ROM with all sources and detectors.}
	\end{subfigure}\:\:\:
	\centering
	   \begin{subfigure}{0.45\textwidth}
		\centering
		\includegraphics[width=0.95\textwidth]{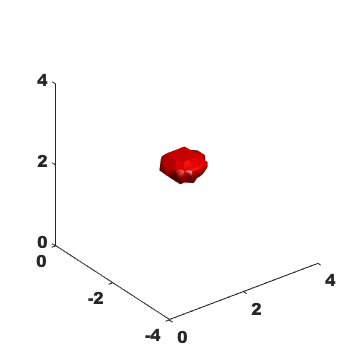}
		\subcaption{Reconstruction using the ROM with randomized sources and detectors, $\ell_s=50$.}
	\end{subfigure}
\caption{Results for Experiment 2. Reconstruction of a test anomaly on a $32\times 32 \times 32$ mesh with 225 sources and 225 detectors. We use 4 frequencies and 27 CSRBFs. To build the ROM, we use 3 parameter sample points.}
\label{Fig:3DExp}
\end{figure}

\section{Conclusions and Future Work}\label{sec:conc}

While model reduction provides a way to drastically reduce the cost of
many expensive linear solves in nonlinear parametric inversion \cite{StuGuKilChatBeatOCon}, the standard method to compute a global basis for our problem still poses a high computational cost. This cost has two components.
First, we solve many large linear systems to compute a candidate basis. Second, since we typically need a relatively high order ROM for robustness, the rank revealing factorization of the resulting large candidate basis
can be fairly expensive.

Our numerical experiments show that we can get
similar quality results with candidate bases with
substantially fewer columns.
Our theoretical analysis, which extends
earlier work in this area, gives an analytical justification for these numerical observations in the setting of interpolatory model reduction for the DOT problem.

We show that using randomized sources and detectors to build the global basis can substantially reduce the cost. Our experiments show that even for a small 3D problem, the number of large linear solves can be reduced substantially, while obtaining similar quality reconstructions. For larger problems with many more sources and detectors and multiple frequencies, we expect much larger gains.

This paper points to several interesting theoretical questions to explore further.
First, would computing optimal tangential interpolation directions lead to essentially better reduced order models than using randomization, with lower order reduced models?
Second, while we show theoretically that the angles between the spaces spanned by
groups of candidate basis vectors for distinct interpolation points should be small
for small anomalies, we observe that this is also the case for sizeable anomalies.
Hence, we plan to extend our analysis to more general cases.
Third, we would like to
derive the necessary theory to provide good estimates for the numbers of random vectors (simultaneous sources and detectors) to use.
 We plan to address these issues in future
papers.

\section*{Acknowledements} 
The authors thank Misha E. Kilmer for several insightful discussions and the use of the PaLS code \cite{Aghasi_etal11}.

This material is based upon work supported by the National Science
Foundation under Grant No. DMS-1720305. Part of this material is  based upon work supported by the National Science Foundation under Grant No. DMS-1439786 and by the Simons Foundation Grant No. 507536 while Gugercin was in residence at the Institute for Computational and Experimental Research in Mathematics in Providence, RI, during the "Model and dimension reduction in uncertain and dynamic systems" program.

\bibliographystyle{abbrv}
\bibliography{ms}

\end{document}